\documentclass[11pt,a4paper]{article}
\usepackage[body={155mm,230mm},top=34mm]{geometry}
\usepackage{bm}
\usepackage{epsfig}
\usepackage{amsmath}
\usepackage{amssymb}
\usepackage{amsthm}
\usepackage{mathrsfs}
\usepackage{epsfig}
\usepackage{verbatim}
\usepackage{url}
\usepackage{color}
\usepackage{hyperref,enumitem}
\usepackage{authblk}

\usepackage{accents}

\newtheorem{Rem}{Remark}

\newtheorem{theorem}{Theorem}

\newtheorem{lemma}{Lemma}

\newtheorem{definition}{Definition}
\newtheorem{remark}{Remark}

\newtheorem{assumption}{Assumption}
\newtheorem{proposition}{Proposition}

\newcommand{\what}{\widehat}

\def \eps{\varepsilon}

\def \bR{\mathbb{R}}

\def \bL{\mathbb{L}}
\def \bE{\mathbb{E}}
\def \bF{\mathcal{F}}
\def \bF{\mathbb{F}}
\def \cY{\mathcal{Y}}
\def \bH{\mathbb{H}}

\def \bS{\mathbb{S}}

\def \cE{\mathcal{E}}

\def \cZ{\mathcal{Z}}

\def \cF{\mathcal{F}}

\def \cY{\mathcal{Y}}

\def \ds{\displaystyle}

\def \tpi{\widetilde{\pi}}

\def \bP{\mathbb{P}}
\def \1{\mathbf{1}}

\newcommand{\Ymin}{Y^{\min}}
\newcommand{\Zmin}{Z^{\min}}
\newcommand{\Pmin}{\psi^{\min}}
\newcommand{\Mmin}{M^{\min}}

\title{
Backward Stochastic Differential Equations with Non-Markovian
Singular Terminal Conditions with General Driver and Filtration\footnote{This work is supported by TUBITAK 
(The Scientific and Technological Research Council of Turkey) through project number 118F163. We are grateful for this support.}
}

\author[1]{Mahdi Ahmadi\thanks{mahdi.ahmadi@metu.edu.tr}}
\author[2]{Alexandre Popier\thanks{alexandre.popier@univ-lemans.fr}}
\author[3]{Ali Devin Sezer\thanks{devin@metu.edu.tr}}
\affil[1]{\small Institute of Applied Mathematics, Middle East Technical University}
\affil[2]{\small Laboratoire Manceau de Math\'ematiques, Le Mans Universit\'e, Avenue O.~Messiaen, 72085~Le~Mans cedex 9, France}
\affil[3]{\small Institute of Applied Mathematics, Middle East Technical University}

\date{\today}

\begin{document}
\maketitle
\begin{abstract}
We consider a class of Backward Stochastic Differential Equations with superlinear driver process $f$
adapted to a filtration supporting at least a $d$ dimensional Brownian motion and a Poisson random measure on ${\mathbb R}^m- \{0\}.$
We consider the following class of terminal conditions
$\xi_1 =  \infty \cdot {\bm 1}_{\{\tau_1 \le T\}}$ where $\tau_1$ is any stopping time with a bounded density in a neighborhood
of $T$ and
$\xi_2 = \infty \cdot {\bm 1}_{A_T}$ where $A_t$, $t \in [0,T]$ is a decreasing sequence of events
adapted to the filtration ${\mathcal F}_t$ that is continuous in probability at $T$.
A special case for $\xi_2$ is $A_T = \{\tau_2 > T\}$ where $\tau_2$ is any stopping
time such that $P(\tau_2 =T) =0.$
In this setting we prove that the minimal supersolutions of the BSDE are in fact solutions, i.e., they attain
almost surely their terminal values.
We further show that the first exit time from a time varying domain
of a $d$-dimensional diffusion process driven by the Brownian motion
with strongly elliptic covariance matrix does have a continuous density; therefore such exit times can be used
as $\tau_1$ and $\tau_2$ to define the terminal conditions $\xi_1$ and $\xi_2.$
The proof of existence of the density is based on the classical Green's functions for the associated PDE.
\end{abstract}

\section{Introduction and Definitions}
A stochastic differential equation with a prescribed terminal condition is called a backward stochastic differential 
equation (BSDE).
If a terminal condition can take the value $+\infty$ it is said to be singular.
BSDE with singular terminal conditions has received considerable attention
at least since \cite{popi:06}. 
They generalize diffusion-reaction partial differential equations (PDE) 
where the singularity of the terminal condition of the BSDE corresponds to singularities
in the final trace of the solution of the PDE (see \cite{grae:hors:sere:18,popi:06,popi:07,popi:16} and \cite{marc:vero:99}).
Moreover BSDE with a singularity at time $T$ are a key tool in the solution of optimal stochastic control problems with terminal constraints 
(see \cite{anki:jean:krus:13,grae:hors:sere:18,krus:popi:15} and the references therein). 
This type of control problem can be interpreted as an optimal liquidation problem in finance (see the preceding references and \cite{guea:16} for an overview).
Given a BSDE with a terminal condition $\xi$ at $T$, a process $Y$ satisfying the BSDE is said to be supersolution
if
\begin{equation} \label{eq:term_cond_super_sol}
\liminf_{t\to T} Y_t \geq \xi
\end{equation}
holds almost surely; $Y$ is said to be minimal if every other supersolution dominates it.
We say $Y$ solves the BSDE with singular terminal condition $\xi$ if
\begin{equation} \label{eq:termcondsol}
\lim_{t\to T} Y_t = \xi;
\end{equation}
i.e., to go from a supersolution to a solution we need to replace the $\liminf$ in \eqref{eq:term_cond_super_sol} with $\lim$ and
$\ge$ with $=$.
The condition \eqref{eq:termcondsol} means that the process $Y$ is continuous at time $T$; for this reason we refer to the problem
of establishing that a candidate solution satisfies \eqref{eq:termcondsol} as the ``continuity problem.''
We further comment on the distinction between {\em solutions}
and minimal supersolutions below, and, as we will explain shortly, minimal supersolutions
and their properties play a key role our analysis.
While minimal supersolutions of BSDE with singular terminal conditions is available in a general setting
(see \cite{krus:popi:15} and subsection \ref{ss:knownresults} below) 
{\em solutions} of BSDE with singular terminal conditions are mostly available
for Markovian
terminal conditions, i.e., terminal conditions which are deterministic functions of an underlying
adapted Markov process; see subsection \ref{ss:knownresults} for a summary of known results.

The first work to solve a BSDE with a non-Markovian singular terminal condition was \cite{krus:popi:seze:18}
treating
the following problem:
\begin{align}\label{e:earlierBSDE}
Y_t &= Y_s - \int_s^t Y_r|Y_r|^{q-1}dr - \int_s^t Z_r dW_r, 0 < s < t < T,\\
Y_T &= \xi, \notag
\end{align}
where $W$ is a single dimensional Brownian motion, 
$\xi = \infty \cdot {\bm 1}_{\{\tau_0 \le T\}}$
or
$\xi = \infty \cdot {\bm 1}_{\{\tau_0 > T\}}$, and $\tau_0$ is the first exit time of $W$ from an interval
$[a,b].$
The goal of the present work is to generalize these results in the following directions: 
\begin{enumerate}
\item Work with a more general filtration supporting a $d$-dimensional Brownian motion and a Poisson random measure,
\item More general driver processes $f$ that is allowed to be an 
${\mathbb F}$-adapted process,
\item For $\xi_1 = \infty \cdot {\bm 1}_{\{\tau \le T\}}$ we allow $\tau$ to be any stopping time whose distribution around
$T$ has a bounded density; we show that the exit time of a multidimensional continuous diffusion process from a time varying
domain satisfies the density condition.
\item Extend $\xi_2 = \infty \cdot {\bm 1}_{\{\tau > T\}}$ to the 
more general terminal condition
$\xi_2 = \infty \cdot {\bm 1}_{A_T} $ where $A_t$, $t \in [0,T]$ is a decreasing sequence of events adapted to ${\mathcal F}_T$ that is left continuous
in probability at $T$.
\end{enumerate}

Let $(\Omega,\cF,\bP,\bF = (\cF_t)_{t\geq 0})$ be a filtered probability space. 
The filtration $\bF$ is assumed to be complete, right continuous,
it supports a $d$ dimensional Brownian motion $W$ and 
a Poisson random measure $\pi$ with intensity $\mu(de)dt$ on the space $\cE \subset \bR^m \setminus \{0\} $. 
The measure $\mu$ is $\sigma$-finite on $\cE$ and satisfies
$$\int_\cE (1\wedge |e|^2) \mu(de) <+\infty.$$
The compensated Poisson random measure $\tpi(de,dt) = \pi(de,dt) - \mu(de) dt$ is a martingale with respect to the filtration $\bF$. 
In this framework we will study the following generalization
of \eqref{e:earlierBSDE}:
\begin{align} 
Y_t  &= Y_s +\int_t^s f(r,Y_r,Z_r,\psi_r) dr - \int_t^s Z_rdW_r -\int_t^s  \int_\cE \psi_r(e) \tpi(de,dr) -\int_t^s d M_r,
\label{eq:bsde}\\
Y_T &= \xi,\label{e:termcond}
\end{align} 
$0 \le t < s < T.$
We call
$(Y,Z,\psi,M)$ a solution to the BSDE (\ref{eq:bsde},\ref{e:termcond}) if
$(Y,Z,\psi,M)$ satisfies (\ref{eq:bsde},\ref{e:termcond})
and $Y$ is continuous at $T$, i.e.,
\[
\lim_{t\rightarrow T} Y_t = Y_T = \xi;
\]

The driver $f$, 
generalizing the deterministic $-y|y|^{q-1}$ appearing
in \eqref{e:earlierBSDE}, is
defined on $\Omega \times [0,T] \times \bR \times \bR^{k} \times (\bL^1_\mu+\bL^2_\mu)$\footnote{For the precise definition of the sum of two Banach spaces, see \cite{krei:petu:seme:82} or the introduction of \cite{krus:popi:17}.}
and for any fixed $y$, $z$, $\psi$, $f(t,y,z,\psi)$ is assumed to be a progressively measurable process;
thanks to apriori bounds and comparison results proved in
\cite{krus:popi:14,krus:popi:15,krus:popi:17}, we are able to work with a very general class
of drivers; to be able to use their bounds and comparison results we will adopt the assumptions these works
make on the filtration
and on the driver, which  are listed in subsection \ref{ss:knownresults} below.

In Section \ref{s:firstcase} we solve the BSDE (\ref{eq:bsde}, \ref{e:termcond})
with\footnote{We define $0\cdot \infty := 0$.}
\[
\xi = \xi_1 = \infty \cdot {\bm 1}_{\{\tau \le T\}},
\]
where $\tau$ is any stopping time whose distribution in a neighborhood of $T$ has 
a bounded density.
In Section \ref{s:secondcase} we treat terminal conditions of the form
\[
\xi = \xi_2 = \infty \cdot {\bm 1}_{A_T},
\]
where $A_t$ is a decreasing left continuous sequence of events
adapted to our filtration; a special case is $A_T = \{\tau > T \}$ where $\tau$ is a stopping
time with ${\mathbb P}(\tau =T) =0.$

We know from \cite{krus:popi:15} that the BSDE \eqref{eq:bsde} has a minimal supersolution $Y_t^{\min}$
with terminal condition $\xi_1.$ The goal of Section \ref{s:firstcase} is to prove that $Y_t^{\min}$ is continuous at
$T$ and has $\xi_1$ as its limit- this implies that the supersolution is indeed a solution.
Let $Y^\infty$ be the solution of \eqref{eq:bsde} with terminal condition $\xi=\infty$ identically.
The main idea in establishing the continuity of the minimal supersolution is to use the process
$t\mapsto {\mathbb E}[Y^\infty_{\tau} {\bm 1}_{\tau \le T} | {\mathcal F}_{\{t\wedge \tau\}}]$ as an upperbound.
The proof that the upperbound process is well defined
involves two ingredients 1) the fact that $\tau$ has a density and 2) apriori upperbounds on $Y^\infty$
derived in \cite{krus:popi:15}. Although the approach of \cite{krus:popi:seze:18} is different from the
one outlined above, it uses these ingredients as well, both of which are elementary 
in the setup treated in \cite{krus:popi:seze:18}: there is an explicit formula for the density
of the exit time $\tau_0$ and the process $t\mapsto y_t$ in \cite{krus:popi:seze:18}
corresponding to $Y^\infty$ is deterministic with an elementary formula so no apriori bounds were needed in \cite{krus:popi:seze:18}.

The treatment of $\xi_2$, given in Section \ref{s:secondcase}
is a generalization of the argument given in \cite{krus:popi:seze:18}
dealing with $\infty \cdot 1_{\{\tau_0 > T \}}$ where $\tau_0$ is the first time
a one dimensional Brownian motion leaves a bounded interval; the argument in \cite{krus:popi:seze:18}
was based on a reduction to PDE whereas in the present work we will be working directly
with the BSDE. 
The idea is as follows:
we construct two sequences of processes, one increasing and one decreasing
such that the decreasing sequence dominates the increasing one.
The limit of the increasing sequence is our candidate solution (in fact it is exactly the minimal supersolution
of \cite{krus:popi:15} with terminal condition $\xi_2$);
the decreasing sequence is used to prove that the candidate solution
indeed satisfies the terminal condition.
The increasing sequence is the solution of \eqref{eq:bsde} with
terminal condition $Y_T = n \cdot {\bm 1}_{A_T}$;
for the decreasing sequence we solve the same BSDE over the
time interval $[0,T-1/n]$ with terminal condition 
$Y^\infty_{T-1/n} \cdot {\bm 1}_{A_{T-1/n}}.$ That all these
sequences are in the right order will be proved by the comparison
principle for the BSDE \eqref{eq:bsde} derived in \cite{krus:popi:15}.

In Section \ref{s:density} we identify a class of stopping times satisfying the assumptions
made on the stopping times above. The class of these stopping times is defined in terms
of   a diffusion process $X$ driven
by the Brownian motion $W$:
\begin{equation}\label{def:X}
X_t = x_0 + \int_0^t b(s,X_s)ds + \int_0^t \sigma(s,X_s)dW_s,
\end{equation}
where $a=\sigma\sigma'$ is assumed to be uniformly and strictly elliptic and $a$ and $b$ assumed
uniformly Holder continuous; these assumptions are adopted from \cite[page 8]{friedman}.
The initial value $x_0$ takes values in a bounded open set
$D_0$. Define
\[
D=  \cup_{t=0}^T  \{t\} \times D_t \subset {\mathbb R}^{d+1};
\]
$D$ satisfies the assumptions
in \cite{friedman}, see Section \ref{s:density} below.
The class of stopping times identified in this section are exit times of $X$ from the domain $D$:
\begin{equation}\label{d:tau}
\tau \doteq \inf\{t\ge 0: X_t \in D_t^c \}.
\end{equation}
To prove that $\tau$ satisfies the assumptions of Sections \ref{s:firstcase} and \ref{s:secondcase}
it suffices to show that it has a continuous density. Despite the considerable literature on exit
times of diffusions we are not aware of a result in the currently available literature establishing that the exit time $\tau$
of \eqref{d:tau} has a density.
Section \ref{s:density} is devoted to the
derivation of this density; the natural tool for this
is the Green's function of the generator of $X$ derived in
\cite{friedman}.

BSDE of the form \eqref{eq:bsde} with singular terminal conditions correspond to stochastic optimal control problems
with constraints, see \cite{krus:popi:15} and  \cite[Section 4]{krus:popi:seze:18}.
The question of whether a supersolution is really
a solution (i.e., its continuity at $T$) is natural in these optimal control applications. For example,
in the context of optimal liquidation of portfolios, it means that the optimal portfolio 
does not super hedge the penalty cost $\xi$. In \cite{bank:voss:18}, 
a positive answer to this question is a condition for solving the optimal targeting problem.
The terminal conditions $\xi_1$ and $\xi_2$ are also natural from the point of view of optimal
liquidation applications; they correspond to putting conditions on when full liquidation takes place.
For a discussion of this potential application we refer the reader to \cite{krus:popi:seze:18}.

The rest of this introduction lists the assumptions we adopt and the results we will be using from prior works;
it also gives a summary of what is known on the solution of BSDE with singular terminal conditions.
In Section \ref{s:conclusion} we comment on possible future work.
\subsection{Assumptions and results from prior works}\label{ss:knownresults}

Let us first define $\bL^p_\mu=\bL^p(\cE,\mu;\bR)$, the set of measurable functions $\psi : \cE \to \bR$ such that
$$\| \psi \|^p_{\bL^p_\mu} = \int_{\cE} |\psi(e)|^p \mu(de)  < +\infty, \quad
\mbox{and} \quad 
\mathfrak{B}^2_\mu =\begin{cases} \bL^2_\mu & \mbox{if } p \geq 2, \\ \bL^1_\mu+\bL^2_\mu & \mbox{if } p < 2. \end{cases}$$
For the definition of the sum of two Banach spaces, see for example \cite{krei:petu:seme:82}. The introduction of $\mathfrak{B}^2_\mu$ is motivated in \cite{krus:popi:17}. We assume that $f : \Omega \times [0,T] \times \bR \times \bR^{m} \times \mathfrak{B}^2_\mu \to \bR$ is a random measurable function, such that for any $(y,z,\psi)\in  \bR \times \bR^{m} \times \mathfrak{B}^2_\mu$, the process $f(t,y,z,\psi)$ is progressively measurable. For notational convenience we write $f^0_t=f(t,0,0,0)$.

The precise assumptions on the driver $f$, adapted from \cite{krus:popi:15} are as follows:
\begin{enumerate}[label=\textbf{(A\arabic*)}]
\item\label{A1} 
The function $y\mapsto f(t,y,z,\psi)$ is continuous and monotone: there exists $\chi \in \bR$ such that a.s. and for any $t \in [0,T]$ and $z \in \bR^m$ and $\psi \in \mathfrak{B}^2_\mu $
\begin{equation*}
(f(t,y,z,\psi)-f(t,y',z,\psi))(y-y') \leq \chi (y-y')^2.
\end{equation*}
\item \label{A2} 
$\sup_{|y|\leq n} |f(t,y,0,0)-f^0_t| \in L^1((0,T)\times \Omega)$
holds for every $n> 0$.
\item \label{A3}  There exists a progressively measurable process $\kappa = \kappa^{y,z,\psi,\phi} : \Omega \times \bR_+\times \bR^m \times \mathfrak{B}^2_\mu \to \bR$ such that
\begin{equation*}
f(t,y,z,\psi)-f(t,y,z,\phi) \leq \int_\cE (\psi(e)-\phi(e))  \kappa^{y,z,\psi,\phi}_t(e)  \mu(de)
\end{equation*}
with $\bP\otimes Leb \otimes \mu$-a.e. for any $(y,z,\psi,\phi)$, $-1 \leq \kappa^{y,z,\psi,\phi}_t(e)$
and $|\kappa^{y,\psi,\phi}_t(e)| \leq \vartheta(e)$ where $\vartheta$ belongs to the dual space of $\mathfrak{B}^2_\mu$, that is $\bL^2_\mu$ or $\bL^\infty_\mu \cap \bL^2_\mu$. 
\item \label{A4}  There exists a constant $L$ such that
\begin{equation*}
|f(t,y,z,\psi)-f(t,y,z',\psi)| \leq L (z-z')
\end{equation*} for any $(t,y,z,z',\psi)$.
\end{enumerate}
The set of conditions {\bf (A)} guarantees the existence and uniqueness of the solution of the BSDE \eqref{eq:bsde} and \eqref{e:termcond} if 
$$\bE \left[ |\xi|^p + \left( \int_0^T |f^0_t| dt \right)^p \right] <+\infty.$$
(see \cite{krus:popi:17,krus:popi:14} and the references therein). 

A key tool for BSDEs is the comparison principle which ensures that if $\xi^1 \leq \xi^2$ a.s., if we can compare the generators $f^1\leq f^2$ along one solution and if the drivers satisfy the conditions {\bf (A)}, then the solutions can be compared: a.s. $Y^1 \leq Y^2$. See among others \cite[Section 3.2]{delo:13}, \cite[Proposition 4]{krus:popi:14} or \cite[Section 5.3.6]{pard:rasc:14}. 

A second set of assumptions are needed
to control the growth of the process $Y$ when the terminal condition
can take the value $+\infty$; these assumptions generalize the
superlinearity of $y\mapsto y|y|^{q-1}$ in \eqref{e:earlierBSDE} and
are adapted from \cite{krus:popi:15}:
\begin{enumerate}[label=\textbf{(C\arabic*)}]
\item\label{C1} 
There exists a constant $q > 1$ and a positive process $\eta$ such that for any $y \geq 0$
\begin{equation*}
f(t,y,z,\psi)\leq -\frac{y}{\eta_t}|y|^{q-1} + f(t,0,z,\psi).
\end{equation*}

\item\label{C2} 
There exists some $\ell > 1$ such that
\begin{equation*}
\bE \int_0^T \left[ \left( \eta_s\right)^{\ell(p-1)}\right]  ds < +\infty
\end{equation*}
where $p$ is the H\"older conjugate of $q$.

\item\label{C3} The parameter $\vartheta$ of \ref{A3} satisfies: for any $\varpi > 2$
\begin{equation*}
\int_\cE |\vartheta(e)|^{\varpi} \mu(de) < +\infty.
\end{equation*}

\item\label{C4} 
On $f^0$, we suppose that a.s. for any $t\in [0,T]$
\begin{equation*}
f_t^0 \ge 0,\qquad \bE \int_0^T \left( f^0_s\right)^\ell ds <+\infty.
\end{equation*}
\end{enumerate}

We further suppose that the generator $(t,y) \mapsto -y|y|^{q-1}/\eta_t$ satisfies 
the {\bf (A)} assumptions, 
which means that $\eta$ satisfies:
\begin{equation} \label{eq:cond_1_over_eta}
\bE \int_0^T \frac{1}{\eta_t} dt < +\infty.
\end{equation}
\begin{remark}[On Assumption \ref{C3}]
In fact it is sufficient to assume that $\vartheta$ belongs to some $\bL^{\rho}_\mu$ for $\rho$ large enough. But it yields to some cumbersome conditions on $\ell$ and $q$ in Theorem \ref{thm:main_result_first_case}. 
\end{remark}

\begin{remark}[On Condition \ref{C4}]
 The work \cite{krus:popi:15} introduces an integrability assumption on $(f_t^0)^- = \max(-f_t^0,0)$ and on $(f^0_t)^+$ (see conditions A4 and A6 in \cite{krus:popi:15}). Hence {\rm \ref{C4}} is stronger. The sign hypothesis could be easily ignore but it would lead to extra technical considerations, which make the presentation of the results heavy. 
 
In \cite[Section 3.1]{popi:16}, it also proved that without the integrability condition {\rm \ref{C4}} on $f^0$, the minimal solution $\Ymin$ can explode a.s. at time $T$ (more precisely, if the non negative process $f^0$ is not in $L^1([0,T]\times \Omega)$). In other words the integrability condition A.6 on $f^0$ in \cite{krus:popi:15} is sufficient to obtain a minimal supersolution but cannot avoid an a.s. explosion at time $T$. 
\end{remark}

From \cite[Theorem 1]{krus:popi:15}, under the setting of conditions {\bf (A)} and {\bf (C)}, and if {\it the filtration is left-continuous at time $T$}, we know that there exists a process $(Y,Z,\psi,M)$ which is a minimal supersolution to the BSDE \eqref{eq:bsde} with singular terminal condition $Y_T = \xi \geq 0$ in the sense that:
\begin{enumerate}
\item for all $t<T$:
$$\bE \left( \sup_{s\in [0,t]} |Y_{s}|^\ell +\left(  \int_0^{t}  |Z_s|^2 ds \right)^{\ell/2}  + \left(  \int_0^{t} \int_\cE |\psi_s(e)|^2 \pi(de,ds)\right)^{\ell/2} + [M ]^{\ell/2}_{t} \right) < +\infty;$$
\item $Y$ is non negative;
\item for all $0\leq s \leq t<T$:
\begin{eqnarray*}
Y_{s}  =  Y_{t } + \int_{s }^{t}  f(u,Y_u,Z_u,\psi_u)du -  \int_{s }^{t}  Z_s dW_s -  \int_{s }^{t} \int_\cE \psi_u(e) \tpi(de,du) - \int_{s}^{t } dM_u.
\end{eqnarray*}
\item The terminal condition \eqref{e:termcond} becomes: a.s. 
\begin{equation} \label{eq:term_cond_super_sol}
 \ds \liminf_{t \to T} Y_{s } \geq \xi.
 \end{equation}
\item For any other supersolution $(Y',Z',\psi',M')$ satisfying the first four properties, we have $Y_t\le Y'_t$ a.s.\ for any $t\in [0,T)$.
\end{enumerate}

As in \cite{krus:popi:seze:18}, we denote this minimal supersolution by $(\Ymin,\Zmin,\Pmin,\Mmin)$. Let us recall that the construction is done by approximation. We consider $(Y^{(k)},Z^{(k)},\psi^{(k)},M^{(k)})$ the unique solution of the BSDE \eqref{eq:bsde} and \eqref{e:termcond} with truncated parameters, namely the terminal condition $\xi \wedge k$ and the driver 
$$f^k(t,y,z,\psi) = \left[ f(t,y,z,\psi)-f^0_t \right] + (f^0_t \wedge k).$$ 
From the comparison principle, the sequence $Y^{(k)}$ is non decreasing and converges to $\Ymin$: a.s. for any $t\in [0,T]$
$$\lim_{k \to +\infty} Y^{(k)}_t = \Ymin_t.$$
The sequence $(Z^{(k)},\psi^{(k)},M^{(k)})$ converges to $(\Zmin,\Pmin,\Mmin)$: for any $0\leq t < T$
\begin{eqnarray*}
&&\lim_{k \to +\infty} \bE \left[  \left(  \int_0^t  | Z^{(k)}_u - \Zmin_u|^2 du \right)^{\ell/2} + \left(  \int_0^t  \int_\cE | \psi^{(k)}_u(e) - \Pmin_u(e) |^2\pi(de,du) \right)^{\ell/2} \right. \\
&&\qquad \qquad + \left( [M^{(k)}-\Mmin]_t \right)^{\ell/2} \bigg] =0.
\end{eqnarray*}

Finally following the arguments of the work \cite{krus:popi:15}, we can prove the following a priori upper estimate on the supersolution: for any $1 < \ell' \leq \ell$, 
\begin{equation}\label{eq:a_priori_estimate_Y_L}
\Ymin_t \leq \frac{K_{\vartheta,L,\ell'}}{(T-t)^{p-\frac{\ell-\ell'}{\ell \ell'}}} \left[ \bE \left( \ \int_t^{T} \left( \left( (p-1)\eta_s\right)^{p-1} + (T-s)^p (f^0_s)^+ \right)^{\ell} ds \bigg| \cF_t\right) \right]^{1/\ell}
\end{equation}
where $K_{\vartheta,L,\ell'}$ is a constant depending only on $\vartheta$, $L$ and $\ell'$. This estimate is valid for any terminal value $\xi$. The proof of this estimate is postponed in the Appendix. 

\subsubsection*{When is a supersolution a solution- continuity at $T$}
The questions that the present work focuses on related to the results summarized above, in the context
of the BSDE \eqref{eq:bsde} and the terminal conditions $\xi_1$ and $\xi_2$ are
\begin{enumerate}
\item Does the limit $\lim_{t\rightarrow T} Y_t^{\min}$ exist?
\item Can the inequality \eqref{eq:term_cond_super_sol} be an equality (if the filtration is left-continuous at time $T$), i.e., is the supersolution $Y_t^{\min}$ in fact a solution?
\end{enumerate}
Let us summarize the known results about these questions in the currently available literature.
The existence of a limit at time $T$ is proved under a structural condition on the generator $f$ (\cite[Theorem 3.1]{popi:16}). Roughly speaking it is proved that $Y$ is a non linear continuous transform of a non negative supermartingale. 

The second question is addressed in \cite{popi:06, popi:16, krus:popi:seze:18, maru:popi:19}. 
In the first two papers \cite{popi:06, popi:16}, the terminal condition $\xi$ is supposed to be Markovian\footnote{No additional assumption is supposed on $f$, that is the setting is only half-Markovian.}, that is $\xi = g(X_T)$, where $X$ is given by \eqref{def:X}\footnote{A jump component driven by the Poisson random measure could be added in the case.}. In \cite{maru:popi:19}, $\xi$ is given by a smooth functional (in the sense of Dupire \cite{dupi:09,cont:four:10,cont:16}) on the paths of $X$. In these three papers, the proof is based on the It\^o formula and on a suitable control on $Z$ and $\psi$, which yields to a condition on $q$ in \ref{C1}, namely $q$ is essentially supposed to be greater than $3$. 

The work \cite{krus:popi:seze:18} was a first attempt to obtain a positive answer to these questions in a non Markovian setting on $\xi$. 
This work 
obtains the continuity of $Y$ at time $T$ with $q>2$ in the first case $\xi_1$ and with $q>1$ in the second case $\xi_2$, which relaxes the assumption on $q$ imposed in \cite{popi:06, popi:16, maru:popi:19}. 
The aim of this paper is to extend this work in the directions indicated
above.

\section{Terminal condition $\xi_1$}\label{s:firstcase}
The goal of this section is to solve the BSDE \eqref{eq:bsde} with terminal
condition $\xi_1 = \infty \cdot {\bm 1}_{\{ \tau \le T \}}$ where $\tau$
is any stopping time whose distribution in a neighborhood of $T$ has a bounded density.
We will see in Section \ref{s:density} below that
first exit times from time varying domains of multidimensional diffusions driven by $W$ satisfy
this condition.
Another simple example is provided by jump times of compound Poisson processes, which are Erlang distributed and they
evidently have densities.

Let $Y^{(k)}$ be the solution of the BSDE \eqref{eq:bsde} with terminal condition
\[
Y^{(k)}_T = \xi \wedge k=  k \cdot {\bm 1}_{\{\tau \le T\}}.
\]
The minimal supersolution of \eqref{eq:bsde}, by definition, is
\[
Y^{\min}_t = \lim_{k\rightarrow \infty} Y^{(k)}_t.
\]
We will construct our solution by showing that $Y^{\min}$ is in fact a solution, i.e., it satisfies 
\begin{equation}\label{e:contcond}
\lim_{t\rightarrow T} Y^{\min}_t = \xi_1.
\end{equation}
The results in \cite{krus:popi:15} imply that \eqref{e:contcond} holds
for $\xi_1 = \infty.$ Therefore, it suffices to show \eqref{e:contcond}
over the event $\{\tau > T \}$ where the right side of \eqref{e:contcond}
is $0$.
We will do so by constructing a positive upperbound process 
$Y^{\infty,u}$ on $Y^{\min}$ 
that converges to $0$ over the same event.
Recall that we suppose that the set of conditions {\bf (A)} and {\bf (C)} hold. Let $Y^\infty$ be the minimal supersolution of \eqref{eq:bsde} with terminal condition $Y_T = \infty$ (if $f(y)=-y|y|^{q-1}$, then 
$Y^\infty_t = ((q-1)(T-t))^{-\frac{1}{q-1}}$).
Define
\[
\xi^{(\tau)}_1 \doteq 
{\bm 1}_{\{\tau < T\}} Y^{\infty}_{\tau}.
\]
The upperbound process $Y^{\infty,u}$ is defined as the solution of the 
BSDE with the terminal value $\xi^{(1)}_\tau=Y^{\infty}_\tau \mathbf 1_{\tau \leq T}$ at the random time $\tau\wedge T$ and the (linear in $y$) generator 
\[
g(t,y,z,\psi) = \chi y +  f(t,0,z,\psi).
\]
For this to be well defined we need the following lemma:
\begin{lemma} \label{lmm:integrability_upper_bound}
If the distribution of $\tau$ in a neighborhood of $T$ has a bounded density
 and if $\ell > 2$ and $q > 2 - \dfrac{2}{\ell-2}$, then there exists some $\varrho>1$
\[
{\mathbb E}_{(x,t)}[ (\xi^{(\tau)}_1)^\varrho ] < \infty.
\]
\end{lemma}
\begin{proof}
The assumptions \ref{C2} and \ref{C4} imply that
\[
M_t \doteq
{\mathbb E}\left[
\int_0^{T} \left( \left( (p-1)\eta_s\right)^{p-1} + (T-s)^p (f^0_s)^+ \right)^{\ell} ds \bigg| \cF_t \right]
\]
is a well defined nonnegative martingale. The hypotheses $\eta_t >0$ and $f^0_t \ge 0$ imply
\[
\int_0^{T} \left( \left( (p-1)\eta_s\right)^{p-1} + (T-s)^p (f^0_s)^+ \right)^{\ell} ds \ge 
\int_t^{T} \left( \left( (p-1)\eta_s\right)^{p-1} + (T-s)^p (f^0_s)^+ \right)^{\ell} ds.
\]
This and the aprori bound \eqref{eq:a_priori_estimate_Y_L} on $Y^\infty$ 
imply for any $1 < \varrho < \ell$
\begin{align*}
{\mathbb E}_{(x,t)}[ {\bm 1}_{\{\tau < T\}} (Y^{\infty}_{\tau})^\varrho] 
&\le
{\mathbb E}_{(x,t)}\left[ {\bm 1}_{\{\tau < T\}}
\frac{K^\varrho_{\vartheta,L}}{(T\wedge \tau -t)^{\hat p\varrho}} M_{\tau \wedge T}^{\frac{\varrho}{\ell}} \right] \\
& \leq K^\varrho_{\vartheta,L}{\mathbb E}_{(x,t)}\left[ {\bm 1}_{\{\tau < T\}}
\frac{1}{(T\wedge \tau -t)^{\kappa}} \right]^{\frac{\ell - \varrho}{\ell}}
{\mathbb E}_{(x,t)}\bigg[ M_{\tau \wedge T} \bigg]^{\frac{\varrho}{\ell}}
\end{align*}
where 
$$\kappa = \dfrac{\hat p \varrho \ell}{\ell-\varrho}, \quad \hat p = p -\frac{\ell-\ell'}{\ell \ell'},$$ 
and where we used the H\"older inequality since $\varrho < \ell$. 

$\frac{1}{(T\wedge \tau -t)^{\kappa}}$ is bounded away from $T$; therefore, to show 
${\mathbb E}_{(x,t)}\left[ {\bm 1}_{\{\tau < T\}} \frac{1}{(T\wedge \tau -t)^{\kappa}} \right] < \infty$,
it suffices to show
${\mathbb E}_{(x,t)}\left[ {\bm 1}_{\{T-\delta < \tau < T\}} \frac{1}{(T\wedge \tau -t)^{\kappa}} \right] < \infty$
for some $\delta > 0.$ We have assumed that the distribution of $\tau$ in a neighborhood of $T$ has a bounded
density, which we will denote by $f^\tau(t,u).$ Then:
\[
{\mathbb E}_{(x,t)}\left[ {\bm 1}_{\{T-\delta < \tau < T\}} \frac{1}{(T\wedge \tau -t)^{\kappa}} \right] 
 = \int_{T-\delta}^T  \frac{1}{(u -t)^{\kappa}} f^\tau(t,u) du,
\]
for some $\delta > 0.$
The boundedness of $f^\tau$ implies that
we obtain the desired result if $\kappa < 1$, that is if
$$p < \frac{\ell-\ell'}{\ell \ell'} + \dfrac{\ell-\varrho}{ \varrho \ell}.$$
The right side is maximal for $\ell'=\varrho=1$. Recall that $p> 1$. Hence we need that $\ell > 2$ and if $q > 2 + \frac{2}{\ell-2}$, then $p < 2\frac{\ell-1}{\ell}$. We can find $\varrho > 1$ and $\ell'>1$ such that the desired inequality holds. 
\end{proof}

\begin{remark}{\em
In \cite{krus:popi:seze:18}, the coefficients are bounded, that is, we can take $\ell=+\infty$ and we get back the condition $q>2$. 
}
\end{remark}

The driver $g$ satisfies all conditions {\bf (A)}. Moreover the terminal time $\tau \wedge T$ is bounded. Hence we apply \cite[Theorem 3]{krus:popi:14,krus:popi:17} and ensure the existence and the uniqueness of the solution $(Y^{\infty,u},Z^{\infty,u},\psi^{\infty,u},M^{\infty,u})$ such that for any $t \in [0,T]$
\begin{eqnarray*}
&& \bE \left[ |Y^{\infty,u}_{t\wedge \tau}|^\varrho + \int_0^{\tau\wedge T} |Y^{\infty,u}_{s}|^\varrho ds +\left(  \int_0^{\tau\wedge T}  |Z^{\infty,u}_s|^2 ds \right)^{\varrho/2}  \right. \\
&& \qquad \left. + \left(  \int_0^{\tau\wedge T} \int_\cE |\psi^{\infty,u}_s(e)|^2 \pi(de,ds)\right)^{\varrho/2}  + \left[M^{\infty,u} \right]^{\varrho/2}_{\tau\wedge T} \right] < +\infty.
\end{eqnarray*}
Note that if $f^0\equiv 0$ and $f$ does not depend on $z$ and $\psi$, then 
\[
Y^{\infty,u}_t = {\mathbb E}[ e^{\chi (\tau-t)} Y_{\tau}^\infty {\bm 1}_{\{\tau < T\}} | {\mathcal F}_t].
\]
We next prove that $Y^{\infty,u}$ does serve as an upperbound on $Y^{(k)}$:
\begin{lemma}\label{l:Ykupperbound}
$Y^{(k)}$ admits upper bound
\[
Y^{(k)}_t \le Y^{\infty,u}_t
\]
a.s.  on the random interval $[\![0,\tau\wedge T]\!]$ 
\end{lemma}
\begin{proof}
The minimal solution $Y^\infty$ is constructed by approximation and for any $n \geq k$, we have: $k \cdot {\bm 1}_{\{\tau \le T\}} \leq n$ a.s. By the comparison principle for BSDEs, a.s. for any $t\in [0,T]$: $Y^{(k)}_t \leq Y^\infty_t.$ Hence a.s. 
$$Y^{(k)}_{\tau \wedge T} = Y^{(k)}_\tau\mathbf 1_{\tau \leq T} \leq Y^{\infty}_\tau \mathbf 1_{\tau \leq T}.$$
Since $Y^{(k)}$ solves the BSDE \eqref{eq:bsde} on the whole interval $[0,T]$, the stopped process $Y^{(k),\tau} = Y^{(k)}_{\cdot \wedge \tau}$ solves the same BSDE on the random interval $[\![0,\tau \wedge \tau]\!]$. 

Now $Y^{\infty,u}$ is the solution of the BSDE with the terminal value $\xi^{(1)}_\tau=Y^{\infty}_\tau \mathbf 1_{\tau \leq T}$ at the random time $\tau\wedge T$ and the generator 
$$g(t,y,z,\psi) =\chi y + f(t,0,z,\psi) .$$
From the assumptions {\bf (A)} on $f$, for any $y\geq 0$, we have
$$f(t,y,z,\psi) \leq f(t,y,z,\psi) - f(t,0,z,\psi) +  f(t,0,z,\psi) \leq \chi y + f(t,0,z,\psi) =  g(t,y,z,\psi).$$
Note that $Y^{(k)}$ and $Y^\infty$ are non negative. Hence we can compare the drivers and deduce the claimed result by the comparison principle. 
\end{proof}

We now prove that the upperbound process has the continuity property we need at terminal time $T$:
\begin{lemma}\label{l:limitofYuatT}
The upperbound process $Y^{\infty,u}$ satisfies:
\[\lim_{t\to T} Y^{\infty,u}_t = 0.
\]
 a.s. on $\{\tau > T\}$
\end{lemma}
\begin{proof}
Indeed for any $0\leq s\leq t$:
\begin{eqnarray*}
Y^{\infty,u}_{s\wedge \tau\wedge T} & = & Y^{\infty,u}_{t\wedge \tau\wedge T} +\int_{s\wedge \tau\wedge T}^{t\wedge \tau\wedge T}  g(r,Y^{\infty,u}_r, Z^{\infty,u}_r,\psi^{\infty,u}_r) dr \\
& - & \int_{s\wedge \tau\wedge T}^{t\wedge \tau\wedge T} Z^{\infty,u}_rdW_r -\int_{s\wedge \tau\wedge T}^{t\wedge \tau\wedge T}  \int_\cE \psi^{\infty,u}_r(e) \tpi(de,dr) -\int_{s\wedge \tau\wedge T}^{t\wedge \tau\wedge T} d M^{\infty,u}_r.
\end{eqnarray*}
Since $g$ is linear in $y$ and using \ref{A3} and \ref{A4}, we have an explicit upper bound on $Y^{\infty,u}$:
$$Y^{\infty,u}_{t} \leq \bE \left[ \mathcal E_{t, \tau\wedge \tau} Y^{\infty}_\tau \mathbf 1_{\tau \leq T} + \int_t^{\tau\wedge T} \mathcal E_{t,s} f^0_s ds \bigg| \mathcal F_t \right] = \Gamma_t,$$
where for $t\leq s$
$$\mathcal E_{t,s} = \exp\left(\chi(s-t)+ L (W_s-W_t) +\frac{L^2}{2} (s-t) \right)V^{\infty}_{t,s} $$
and $V^\infty$ is the Dol\'eans-Dade exponential:
\begin{equation*}
V^{\infty}_{t,s} = 1+ \int_t^s \int_{\cE}  V^{\infty}_{t,u^-} \kappa^{0,0,\phi^{\infty,u},0}_u(e) \tpi(de,du).
\end{equation*}
From assumptions \ref{C3} and \ref{C4}, together with the integrability property proved in Lemma \ref{lmm:integrability_upper_bound}, we obtain that if $\tau > T$,
$$0\leq \lim_{t\to T} Y^{\infty,u}_{t} \leq \lim_{t\to T} \Gamma_t = 0,$$
which achieves the proof of the lemma.
\end{proof}

Combining the lemmas above we have the main result of this section:
\begin{theorem} \label{thm:main_result_first_case}
Under conditions {\rm {\bf (A)}} and {\rm {\bf (C)}}, if the distribution of the stopping time $\tau$ 
is given by a  bounded density in a neighborhood of $T$, $\ell > 2$ and $q > 2 - \dfrac{2}{\ell-2}$, 
then the minimal supersolution with terminal condition $\xi_1$ satisfies
\begin{equation}\label{e:contatT}
\lim_{t\to T} \Ymin_t  = \xi_1
\end{equation}
almost surely.
\end{theorem}
\begin{proof}
As stated in the beginning of this section it suffices to prove \eqref{e:contatT} over the event $\{ \tau > T \}$
where $\xi_1 = 0.$ By our assumptions on the driver $f$ $Y^{(k)}$ is nonnegative; this and Lemma \ref{l:Ykupperbound} gives
\[
0 \le Y^{(k)}_t \le Y^{\infty,u}_t.
\]
On the other hand, by Lemma \ref{l:limitofYuatT}, the limit as $t\rightarrow T$ of the right side is $0$ over the event
$\{\tau > T\}.$ These imply \eqref{e:contatT}.
\end{proof}
This result generalizes the continuity result \cite[Theorem 2.1]{krus:popi:seze:18}. If the setting of this former result was less general, we were able to describe precisely the minimal solution, namely that it is  obtained by pasting two processes at time $\tau$.
In the following paragraphs we discuss the possibility of defining a solution to the BSDE by this pasting method;
the main point is this: the presence of the  orthogonal martingale $M$ complicates this approach; but if the filtration
is assumed to be generated by $W$ and $\pi$ alone then the same technique can be used in the present setting as well.

Let $Y^{1,\tau}$ be the solution of the BSDE \eqref{eq:bsde} in the time interval $[0,\tau \wedge T]$ with terminal condition $\xi_1^{(\tau)}$ (again we can apply \cite[Theorem 3]{krus:popi:14,krus:popi:17} as for $Y^{\infty,u}$). Following the idea of \cite[Theorem 2.1]{krus:popi:seze:18}, let us define
\[
Y_t^{(1)}\doteq \begin{cases}
Y_t^{1,\tau}, & t \le \tau \wedge T \\
Y_t^{\infty} & \tau < t \le T,
\end{cases}
\]
where we assume that $\tau$ is an $\mathbb F^W$ stopping time, that is it just depends on the paths of $W$, and is predictable 
(exit times of Section \ref{s:density} are a particular case). 
The jump times of $Y_t^{1,\tau}$ and of $Y^{\infty}$ coincide with the jump times of the Poisson random measure or of the orthogonal martingale component. A consequence of the Meyer theorem (see \cite[Chapter 3, Theorem 4]{prot:04}) implies that the jump times of $\pi$ are totally inaccessible, hence a.s. cannot be equal to $\tau$. However we cannot exclude that the orthogonal martingale may have a jump at time $\tau$. The second issue is the definition of the martingale part $(Z,\psi,M)$. For the first two components, we can easily paste them together
 \[
Z_t^{(1)}\doteq \begin{cases}
Z_t^{1,\tau}, & t \le \tau \wedge T \\
Z_t^{\infty} & \tau < t \le T.
\end{cases},\qquad \psi_t^{(1)}(e) \doteq \begin{cases}
\psi_t^{1,\tau}(e), & t \le \tau \wedge T \\
\psi_t^{\infty}(e) & \tau < t \le T.
\end{cases}.
\]
Since $\tau$ is predictable, these two processes are also predictable and the stochastic integrals 
$$\int_0^\cdot Z_t^{(1)} dW_t,\quad \int_0^\cdot \int_\cE \psi_t^{(1)}(e) \tpi(de,dt)$$
are well-defined and are local martingales on $[0,T)$. Nonetheless if we define $M^{(1)}$ similarly, we cannot ensure that this process is still a local martingale. For the parts with $Z$ and $\psi$, the local martingale property is due to the representation as a stochastic integral.
Based on these observations we provide the following result on the pasting method under the assumption
that the filtration is generated by $W$ and $\pi$ alone; the approach in the proof of this proposition is the generalization
of the approach used in \cite{krus:popi:seze:18}.
\begin{proposition}
Assume that the filtration is generated by $W$ and $\pi$. Then $Y_t^{(1)}$ solves the BSDE \eqref{eq:bsde} on $[0,T]$
with terminal condition $Y_T = \xi_1$ and satisfies the continuity property at time $T$. Moreover $Y^{(1)} = \Ymin$.
\end{proposition}
\begin{proof}
Since there is no additional martingale $M$ in the definition of $Y^{1,\tau}$ and $Y^\infty$, the resulting process $Y^{(1)}$ is continuous at time $\tau$. 

Now let us fix $s < t < T$. On the set $\{ \tau \leq s\}$, $Y_r^{(1)} = Y^\infty_r$ for any $r \in [s,t]$. Therefore we have 
$$Y^{(1)}_s  =Y^{(1)}_t +\int_s^t  f(r,Y^{(1)}_r,Z^{\infty}_r,\psi^{\infty}_r) dr - \int_s^t Z^\infty_rdW_r -\int_s^t  \int_\cE \psi^\infty_r(e) \tpi(de,dr).$$
The dynamics of $Y^{1,\tau}$ is given by:
\begin{eqnarray*}
Y^{1,\tau}_{s\wedge \tau\wedge T} & = & Y^{1,\tau}_{t\wedge \tau\wedge T} +\int_{s\wedge \tau\wedge T}^{t\wedge \tau\wedge T}  f(r,Y^{1,\tau}_r,Z^{1,\tau}_r,\psi^{1,\tau}_r) dr \\
& - & \int_{s\wedge \tau\wedge T}^{t\wedge \tau\wedge T} Z^{1,\tau}_rdW_r -\int_{s\wedge \tau\wedge T}^{t\wedge \tau\wedge T}  \int_\cE \psi^{1,\tau}_r(e) \tpi(de,dr) .
\end{eqnarray*}
It implies that for $\{\tau \geq t\}$, $Y^{(1)}$ has the required dynamics. Finally for $\{\tau \in (s,t) \}$, we have
\begin{eqnarray*}
Y^{1,\tau}_{s} & = & Y^{1,\tau}_{ \tau} +\int_{s}^{ \tau}  f(r,Y^{1,\tau}_r,Z^{1,\tau}_r,\psi^{1,\tau}_r) dr \\
& - & \int_{s}^{ \tau} Z^{1,\tau}_rdW_r -\int_{s}^{\tau}  \int_\cE \psi^{1,\tau}_r(e) \tpi(de,dr) 
\end{eqnarray*}
and
$$Y^\infty_\tau  =Y^{\infty}_t +\int_\tau^t  f(r,Y^{\infty}_r,Z^{\infty}_r,\psi^{\infty}_r) dr - \int_\tau^t Z^\infty_rdW_r -\int_\tau^t  \int_\cE \psi^\infty_r(e) \tpi(de,dr).$$
By the continuity of $Y^{(1)}$ at time $\tau$, we get the desired dynamics also in this case.

Finally let us show that $Y^{(1)}$ is continuous at time $T$. On the set $\{\tau < T\}$, we have 
$$\lim_{t\to T} Y^{(1)}_t = \liminf_{t\to T} Y^{(1)}_t = \liminf_{t\to T} Y^\infty_t = +\infty.$$
And on $\{\tau \geq T\}$, 
$$\lim_{t\to T} Y^{(1)}_t = \lim_{t\to T} Y^{(1,\tau)}_{t} =\xi_1^{(\tau)} =0. $$
We can conclude that $Y^{(1)}$ satisfies the BSDE \eqref{eq:bsde} on $[0,T]$ with terminal condition $Y_T = \xi_1$ and is continuous at time $T$.

From the minimality of $\Ymin$, we have immediately that $\Ymin_t \leq Y^{(1)}_t$, a.s. for any $t \in [0,T]$. To obtain the converse inequality, let us define 
\[
Y_t^{(1),n}\doteq \begin{cases}
Y_t^{1,\tau,n}, & t \le \tau \wedge T \\
Y_t^{n} & \tau < t \le T
\end{cases}
\]
where $Y^n$ (resp. $Y^{1,\tau,n}$) is the solution of the BSDE \eqref{eq:bsde} on $[0,T]$ (resp. on $[0,\tau\wedge T]$) with terminal condition $n$ (resp. $Y^n_\tau \mathbf 1_{\tau \leq T}$). Then we have that for any $k\geq n$, $Y^{(1),n} \leq Y^{(k)} \leq \Ymin$. By construction of $Y^\infty$, $Y^{(1),n}$ converges to $Y^{(1)}$. Therefore we conclude that $Y^{(1),\tau} = \Ymin$ and this achieves the proof of the Proposition.
\end{proof}

\section{Terminal condition $\xi_2$}\label{s:secondcase}

The goal of this section is to prove the continuity of the minimal
supersolution for the terminal condition 
\[
\xi = \xi_2 = \infty \cdot {\bm 1}_{A_T},
\]
where $A_t$ is a decreasing sequence of events
adapted to our filtration: for any $s \leq t$, $A_t \subset A_s$ and $A_t \in \cF_t$. If $\tau_0$ is a stopping time, the set $A_t= \{\tau_0 > t \}$ provides an example. 
We also assume that:
\begin{enumerate}[label=\textbf{(H\arabic*)}]
\item\label{H1} The sequence is left continuous at time $T$ in probability, i.e.
\begin{equation} \label{eq:cond_left_cont_A_T}
\end{equation}
\item \label{H2} There exists an increasing sequence $(t_n, \ n\in \mathbb N)$ such that $t_n < T$ for any $n$, $\lim_{n\to +\infty} t_n = T$, and such that the filtration $\mathbb F$ is left continuous at time $t_n$ for any $n$. Recall that we already assume left continuity of $\mathbb F$ at time $T$. 
 \end{enumerate}
If $A_t$ is defined as $A_t = \{\tau_0 > t\}$ through a stopping time $\tau_0$, assumption (\ref{H1}) is equivalent to: ${\mathbb P}(\tau_0=T)= 0.$
In particular if $\tau_0$ has a density this condition is satisfied. Therefore, as in the previous section, if $\tau_0$ is the jump time
of an ${\mathbb F}$-adapted compound Poisson process, then it generates a sequence $A_t$ satisfying (\ref{H1}). The same comment
applies to the exit times whose densities are derived in the next section.

\begin{Rem}[On Condition \ref{H2}]
If the filtration $\mathbb F$ is quasi left-continuous, then {\rm \ref{H2}} holds for any sequence $t_n$. In particular our hypothesis is valid if $\mathbb F$ is generated by $W$ and $\pi$. 

The notion of jumps for a filtration has been studied in \cite{jaco:skor:94} (see also \cite[Section 2]{prot:15}). Let us note that we are not able to construct a counter example, that is a filtration such that {\rm \ref{H2}} does not hold.
\end{Rem}
Let us define the random time:
\[
\tau \doteq \inf\{t: \omega \in A_t^c \}.
\]
\begin{lemma}
$\tau$ is a stopping time of the filtration ${\mathbb F}$. If \eqref{eq:cond_left_cont_A_T} holds, then $A_T = \{\tau \geq T\}$.  
\end{lemma}
\begin{proof}
The definition of $\tau$ and $A_{s}^c \searrow A_t^c$ as $s \searrow t$
imply
\begin{equation}\label{e:talet}
\{ \tau \le t \} = \bigcap_{n=1}^\infty  A_{t+ 1/n}^c \in  \bigcap_{n=1}^\infty  \cF_{t+ 1/n}.
\end{equation}
The right continuity of the filtration ${\mathbb F}$ implies
$\{ \tau \le t \} \in {\mathcal F}_t$ i.e., $\tau$ is a stopping time.

Again the definition of $\tau$ and $A_{s}^c \searrow A_T^c$ as $s \searrow T$
imply $\{\tau < T\} = \bigcup_{n=1}^\infty A_{T-1/n}^c$ and 
\[
\{\tau < T\} \subset A^c_T.
\]
The continuity in probability at time $T$ of $(A_t, \ t \geq 0)$ implies
$${\mathbb P}( A^c_T \setminus \{\tau < T\}) =0,$$
which achieves the proof of the lemma. 
\end{proof}
Note that in general: $\{\tau > T\} \subset A_T$. To have the equality, we need a right continuity at time $T$.

Let us denote again by $Y^\infty$ the minimal solution of the BSDE \eqref{eq:bsde} with terminal condition $+\infty$ and:
$$A_{t_n} = A_n, \quad  \mathbf 1_{A_{t_n}}= \chi_n.$$
Let us define $\mathcal Y^n$ as the solution of the following BSDE: for all $t_n\leq t \leq T$
\begin{eqnarray*}
\mathcal Y^n_t  &=& \int_{t}^T \left[(f(s,(1-\chi_n)\mathcal Y^n_s,\mathcal Z^n_s,\mathcal U^n_s) -f^0_s\right] ds + \int_t^T (1-\chi_n) f^0_r dr\\
& -& \int_t^T \mathcal Z^n_s dW_s - \int_t^T \int_\cE \mathcal U^n_s \tpi(de,ds) - \int_t^T d \mathcal M^n_s.
\end{eqnarray*}
Namely we consider the BSDE with terminal value 0 and generator
$$\widetilde f (s,y,z,u)= \left[(f(s,(1-\chi_n)y,z,u) -f^0_s\right] + (1-\chi_n) f^0_s.$$ 
This driver $\widetilde f$ satisfies all assumptions {\bf (A)} and \ref{C4} holds, such that there exists a unique solution to this BSDE such that 
$$\bE \left[ \sup_{t\in [t_n,T]} |\mathcal Y^n_t|^\ell \right] \leq \bE \int_{t_n}^T |f^0_r|^\ell dr .$$ 
Moreover by comparison principle, a.s. for all $t\in [t_n,T]$, $\mathcal Y^n_t \geq 0$. Let us also remark that if $f^0\equiv 0$, then $\mathcal Y^n\equiv 0$.

Define $Y^{\infty,u,n}$ as the solution of the BSDE \eqref{eq:bsde} on $[0,t_n]$ with terminal condition
\[
Y^{\infty,u,n}_{T-1/n} =  \chi_n Y_{t_n}^{\infty} + (1- \chi_n)\mathcal Y^n_{t_n}.
\]
Note that from \eqref{eq:a_priori_estimate_Y_L}, this terminal condition is in $L^\ell(\Omega)$, hence the solution is well-defined on $[0,t_n]$. We extend $Y^{\infty,u,n}$ on the whole interval $[0,T]$: for all $t_n\leq t \leq T$:
$$Y^{\infty,u,n}_t = \chi_n Y^\infty_t + (1-\chi_n) \mathcal Y^n_t.$$
\begin{lemma}
The process $Y^{\infty,u,n}$ satisfies the dynamics of the BSDE \eqref{eq:bsde} on the whole interval $[0,T]$. Moreover a.s. 
$$\lim_{t\to T} Y^{\infty,u,n}_t = \infty \mathbf 1_{A_{T-1/n}}.$$
\end{lemma}
\begin{proof}
Indeed by the very definition of $Y^\infty$, for any $t_n\leq t < s < T$, we have
$$Y^\infty_t  = Y^\infty_s +\int_t^s f(r,Y^\infty_r,Z^\infty_r,\psi^\infty_r) dr - \int_t^s Z^\infty_rdW_r -\int_t^s  \int_\cE \psi^\infty_r(e) \tpi(de,dr) -\int_t^s d M^\infty_r,$$ 
hence multiplying both sides by $\chi_n$, which is $\cF_{t_n}$-measurable, we obtain
\begin{eqnarray*}
\chi_nY^\infty_t & = & \chi_nY^\infty_s +\int_t^s \chi_nf(r,Y^\infty_r,Z^\infty_r,\psi^\infty_r) dr \\
&- & \int_t^s \chi_nZ^\infty_r dW_r -\int_t^s  \int_\cE\chi_n \psi^\infty_r(e) \tpi(de,dr) -\int_t^s d\chi_n M^\infty_r\\
& = &  \chi_nY^\infty_s +\int_t^s \left[ f(r,\chi_nY^\infty_r,\chi_nZ^\infty_r,\chi_n\psi^\infty_r) - f^0_r \right] dr +\int_t^s \chi_n f^0_r dr   \\
&- & \int_t^s \chi_nZ^\infty_r dW_r -\int_t^s  \int_\cE\chi_n \psi^\infty_r(e) \tpi(de,dr) -\int_t^s \chi_n d M^\infty_r.
\end{eqnarray*}
And from the definition of $\mathcal Y^n$, we have
\begin{eqnarray*}
(1-\chi_n) \mathcal Y^n_t  &=& \int_{t}^T(1-\chi_n) \left[(f(s,(1-\chi_n)\mathcal Y^n_s,\mathcal Z^n_s,\mathcal U^n_s) -f^0_s\right] ds + \int_t^T (1-\chi_n) f^0_r dr\\
& -& \int_t^T (1-\chi_n)\mathcal Z^n_s dW_s - \int_t^T \int_\cE(1-\chi_n) \mathcal U^n_s \tpi(de,ds) - \int_t^T (1-\chi_n)d \mathcal M^n_s \\
& =&  \int_{t}^T \left[(f(s,(1-\chi_n)\mathcal Y^n_s,(1-\chi_n)\mathcal Z^n_s,(1-\chi_n)\mathcal U^n_s) -f^0_s\right] ds + \int_t^T (1-\chi_n) f^0_r dr\\
& -& \int_t^T (1-\chi_n)\mathcal Z^n_s dW_s - \int_t^T \int_\cE(1-\chi_n) \mathcal U^n_s \tpi(de,ds) - \int_t^T (1-\chi_n)d \mathcal M^n_s
\end{eqnarray*}
Thereby $Y^{\infty,u,n}$ satisfies the dynamics of the BSDE \eqref{eq:bsde} on $[t_n,T)$. From our assumption \ref{H2}, there is no jump at time $t_n$. Hence $Y^{\infty,u,n}$ is continuous at time $t_n$ and we can define 
$$
Z_t^{\infty,u,n}\doteq \begin{cases}
Z_t^{\infty,u,n}, & t \leq t_n \\
\chi_n Z_t^{\infty} + (1-\chi_n)\cZ^n_t & t_n < t \le T,
\end{cases},$$
$$\psi_t^{\infty,u,n}(e) \doteq \begin{cases}
\psi_t^{\infty,u,n}(e), & t \le t_n \\
\chi_n \psi^\infty_t(e) + (1-\chi_n)\mathcal U^n_t(e) & t_n < t \le T,
\end{cases}.
$$
and 
$$M_t^{\infty,u,n} \doteq \begin{cases}
M_t^{\infty,u,n}, & t < t_n \\
\chi_n M^\infty_t + (1-\chi_n)\mathcal M^n_t& t_n \leq t \le T.
\end{cases}.
$$
Then we have that the process $(Y^{\infty,u,n},Z^{\infty,u,n},\psi^{\infty,u,n},M^{\infty,u,n})$ satisfies the dynamics of the BSDE \eqref{eq:bsde} on the whole interval $[0,T)$ and with the singular terminal value $\infty \mathbf 1_{A_{t_n}}$: a.s. 
$$\lim_{t\to T} Y^{\infty,u,n}_t = \infty \mathbf 1_{A_{t_n}}.$$
The only remaining issue concerns $M^{\infty,u,n}$; indeed it is not clear that it is a martingale on $[0,T)$. However $(Y^{\infty,u,n},Z^{\infty,u,n},\psi^{\infty,u,n},M^{\infty,u,n})$ has the dynamics of the BSDE \eqref{eq:bsde} on the interval $[0,t_{n+1}
]$, with terminal condition $\zeta = Y^{\infty,u,n}_{t_{n+1}} = \chi_n Y^\infty_{t_{n+1}} + (1-\chi_n) \mathcal Y^n_{t_{n+1}}$. This terminal value belongs to $L^\ell(\Omega)$. Hence there exists a unique solution $(y,z,v,m)$ to the BSDE \eqref{eq:bsde} with terminal condition $\zeta$. From uniqueness on $[t_n,t_{n+1}]$, $y=\chi_n Y^{\infty} + (1-\chi_n)\cY^n$ and $m=\chi_n M^\infty + (1-\chi_n)\mathcal M^n$ on this interval. And by uniqueness on $[0,t_n]$ for the BSDE with driver $f$ and terminal condition $y_{t_n}$, $y=Y^{\infty,u,n}$ and $m=M^{\infty,u,n}$ on $[0,t_n]$. Since the martingale $m$ has no jump at time $t_n$ (Hypothesis \ref{H2}), we obtain that $M^{\infty,u,n}$ is a martingale on $[0,t_{n+1}]$ and thus on $[0,T)$. 
\end{proof}

Let us consider $Y^L$ the approximation of the minimal solution $Y$. Let us prove 
\begin{lemma}
A.s. for all $t\in [0,T]$, $L$ and $n$
$$0\leq Y^L_t \leq Y^{u,\infty,n}_{t}.$$
\end{lemma}
\begin{proof}
Fix $L > 0$ and recall that $(Y^L, Z^L,U^L,M^L)$ denotes the solution of truncated BSDE with terminal condition 
$$Y_T^L =\xi\wedge L = L \mathbf 1_{A_T}.$$ 
Set
 $$\what Y_s = Y^{u,\infty,n}_s - Y^{L}_s, \quad \what Z_s = Z^{u,\infty,n}_s - Z^{L}_s, \quad \what U_s(e) = U^{u,\infty,n}_s(e) - U^{L}_s(e), \quad \what M_s = M^{u,\infty,n}_s - M^{L}_s.$$
We have
\begin{eqnarray*}
&& f(t,Y^{u,\infty,n}_t,Z^{u,\infty,n}_t,U^{u,\infty,n}_t) - f(t,Y^L_t,Z^L_t,U^L_t) \\
&& \quad =  -c_t \what Y_t  + b_t \what Z_t + (f(t,Y^L_t,Z^L_t,U^{u,\infty,n}_t) - f(t,Y^L_t,Z^L_t,U^L_t))
\end{eqnarray*}
with
\begin{eqnarray*}
-c_t & =& \frac{f(t,Y^{u,\infty,n}_t,Z^{u,\infty,n}_t,U^{u,\infty,n}_t) - f(t,Y^L_t,Z^{u,\infty,n}_t,U^{u,\infty,n}_t)}{\what Y_t} \1_{\what Y_t \neq 0}
\end{eqnarray*}
and 
\begin{eqnarray*}
b_t & =& \frac{f(t,Y^{u,\infty,n}_t,Z^{u,\infty,n}_t,U^{u,\infty,n}_t) - f(t,Y^L_t,Z^L_t,U^{u,\infty,n}_t)}{\what Y_t} \1_{\what Z_t \neq 0}.
\end{eqnarray*}
Note that from our setting, $-c_t \leq \chi$ and $|b_t| \leq K$. For every $t < T$ the process $(\what Y,\what Z, \what U, \what M)$ solves the BSDE
\begin{eqnarray*}
d\what Y_s & = & \left[ c_s \what Y_s -b_s \what Z_s - (f^0_s - L)^+ - (f(s,Y^L_s,Z^L_s,\psi^{u,\infty,n}_s) - f(s,Y^L_s,,Z^L_s,\psi^L_s)) \right] ds \\
& + & \what Z_s dW_s +  \int_\cZ \what \psi _s(z) \tpi(dz,ds) + d\what M_s
\end{eqnarray*}
on $[0, t]$ with terminal condition $\what Y_t = Y^{u,\infty,n}_t - Y^L_t$. Moreover it holds that 
$$f(s,Y^L_s,Z^L_s,\psi^{u,\infty,n}_s) - f(s,Y^L_s,Z^L_s,\psi^L_s)\geq \int_\cZ  \kappa_s^{Y^L,\psi^L,\psi'} \what \psi_ s(z) \mu(dz).$$
From Lemma 10 in \cite{krus:popi:14}, we have
\begin{eqnarray*}
\what Y_s & \geq  &  \bE \left[\what Y_t \Gamma_{s,t}   + \int_s^t \Gamma_{s,u} (f^0_u - L)^+ du \bigg| \cF_s \right]
\end{eqnarray*}
where $\Gamma_{s,t} = \exp\left( - \int_s^t c_u du +\frac{1}{2}  \int_s^t (b_u)^2 du - \int_s^t b_u dW_u\right) \zeta_{s,t}$ with $\zeta_{s,s}=1$ and
\begin{equation*}
d\zeta_{s,t} = \zeta_{s,t^-}  \int_\cZ \kappa_t^{Y^L,\psi^L,\psi'}\tpi(dz,dt).
\end{equation*}
Our assumptions ensure that $\zeta$ is non negative and belongs to $\bH^k(0,T)$. We have $Y^L_t \leq (1+T)L$ and hence $\what Y_t \geq -(1+T)L$. Thus $\what Y  \Gamma_{s,.}$ is bounded from below by a process in $\bS^m(0,T)$ for some $m>1$. We can apply Fatou's lemma to obtain
\begin{eqnarray*}
\what Y_s & = & \liminf_{t \nearrow T} \bE \left[\what Y_t  \Gamma_{s,t} + \int_s^t \Gamma_{s,u}(f^0_u - L)^+ du \bigg| \cF_s \right]  \geq  \bE \left[  \liminf_{t \nearrow T} (\what Y_t  \Gamma_{s,t})  \bigg| \cF_s \right].
\end{eqnarray*}
The process $(\Gamma_{s,t}, \ s\leq t\leq T)$ is c\`adl\`ag and non negative. Hence a.s.
$$ \liminf_{t \nearrow T} (\what Y_t  \Gamma_{s,t}) = (\liminf_{t \nearrow T} \what Y_t) \Gamma_{s,T^-}.$$
But 
$$\liminf_{t \nearrow T} \what Y_t = \infty \mathbf 1_{A_{T-1/n}} - L \mathbf 1_{A_{T}} \geq 0$$
since $A_T \subset A_{T-1/n}$. 
Finally, $Y^{u,\infty,n}_s \geq Y^L_s$ for any $s \in [0,T]$ and $L\geq 0$. 
\end{proof}

Let us now conclude about the continuity of $Y$ at time $T$. We know now that a.s. 
$$0\leq Y^L_t \leq Y_t \leq Y^{u,\infty,n}_t$$
and we want to prove that for a.e. $\omega \in  A_T^c$, 
$$\lim_{t\to T} Y_t= 0.$$
Recall that 
$$\bP \left(\bigcap_{t<T} A_t \setminus A_T \right) = 0.$$
Let us fix $\omega \in A_T^c$. We can assume (with probability 1) that $\omega$ belongs to $\bigcup_{t<T} A_t^c$, that is there exists $n$ such that $\omega \in A_{T-1/n}^c$. Hence 
$$\limsup_{t\to T} Y_t(\omega) \leq Y^{u,\infty,n}_T(\omega) = 0.$$

\section{Density formula in terms of Green's function}\label{s:density}

As noted in the introduction, 
one of the key ingredients in \cite{krus:popi:seze:18} 
in the analysis of the terminal
condition ${\bm 1}_{\{\tau_0 < T\}}$ was the explicit formula available for the density of $\tau_0$,
the first exit time of the Brownian motion from an interval $(a,b).$ 
The natural framework for the generalization of this formula to higher dimensions is the 
duality between Potential theory, elliptic / parabolic PDE and Diffusion processes \cite{doob2012classical}.
Within this duality the exit times and the distribution of the path of the process up to the exit time
corresponds to Green's functions \cite{oksendal2003stochastic}. The paper \cite{delarue2013first}
uses the connection between hitting times and Green's functions to prove that the exit time of a one
dimensional diffusion from a region has a density. A similar one dimensional computation is also given in \cite{peskir2006optimal}.
Although the term ``Green's function'' doesn't appear in them,
the works \cite{iyengar1985hitting,metzler2010first} compute the Green's function for the Brownian motion
in rectangular domains using the method of images; the work \cite{blanchet2013hitting} extends this to three dimensions.
The work \cite{patie2008first},
represents the distribution of the exit time of a $d$-dimensional diffusion from a fixed domain
as the solution of a parabolic PDE. It identifies a smooth solution to the PDE whose derivative gives the density of the stopping time.
The solution of the same PDE can be expressed in terms of the Green's function derived in the classical PDE book \cite{friedman} by
Friedman for the underlying parabolic PDE.
The same Green's function can be used to prove that exit times of diffusions from domains that vary over time have densities.
Given the duality between Green's functions and exit times, this is a natural result. But we have not been able
to identify a reference in the current literature stating and proving it and therefore give its details in the present work.

The time variable
in \cite{friedman} corresponds to the time to maturity in the present setup. We state all definitions and results 
from \cite{friedman} in terms of the time variable adopted in the present work (which is the one commonly used in the the stochastic
processes framework); therefore, for example, the initial condition of \cite{friedman} becomes the terminal
condition and $t$ derivatives are multiplied by $-$.

Let ${\mathcal L}$ denote the parabolic operator associated with $X$:
\[
{\mathcal L}u \doteq \langle \sigma(x,t), \sigma(x,t) Hu \rangle  + \langle b(x,t), \nabla_x u \rangle + \frac{\partial u}{\partial t},
\]
where $Hu$ is the Hessian  matrix of second derivatives of $u$.
if we define
\[
a = \sigma \sigma'
\]
the first term can also be written as $\langle a, Hu \rangle.$ To be able to use the results in \cite{friedman} we adopt all of the
assumptions it makes on $a$ and $b$, these are listed on \cite[page 8]{friedman}: $a$ is uniformly elliptic; $a$ and $b$ are Holder
continuous.
The formal definition of Green's function is as follows (\cite[page 82]{friedman}):
\begin{definition}
A function $G(x,t,y,s)$ defined and continuous for $(x,t,y,s) \in \bar{D} \times (D \cup B)$, $t < s$ is called
a {\em Green's function} of ${\mathcal L}u = 0$ in $D$ if for any $0\le s \le T $ and for any continuous function $f$ on $D_s$ 
having a compact support the function
\[
u(x,t) = \int_{D_s} G(x,t,y,s)f(y) dy
\]
is a solution of $Lu = 0$ in $D \cap \{ 0 \le t < s \}$ and it satisfies the terminal and boundary conditions
\begin{align*}
\lim_{t\rightarrow s} u(x,t) &= f(x), x \in \overline{D}_s\\
u(x,t) &= 0, (x,t) \in S \cap \{ 0 \le t < s \}.
\end{align*}
\end{definition}

The main result claiming the existence of Green's functions associated with $X$ is \cite[Theorem 16, page 82]{friedman}.
This result is based on the following assumptions on the domain $D$ (listed as conditions
 $\overline{E}$ and $\overline{\overline{E}}$ on \cite[pages 64,65]{friedman}):
\begin{assumption}\label{as:ebar}
For every point $(x,t) \in \overline{S}$ there exists an $(n+1)$-dimensional neighborhood $V$ such that
$V \cap \overline{S}$ can be represented in the form
\[
x_i = h(x_1,...,x_{i-1},x_{i+1},...,x_n,t)
\]
for some $i \in \{1,2,3,...,n\}$, $h$, $D_x h$, $D_x^2h$ and $D_t h$ exist and are Holder continuous (exponent $\alpha$);
$D_xD_th$, $D_t^2h$  exist and are continuous.
\end{assumption}

The Green's function $G$ allows one to compute not just the distribution of the
exit time of $X$ from a fixed domain but from a domain varying in time such as $D$; in fact it allows one to compute
expectations of the form ${\mathbb E}_{(x,t)}[ g(X_s) 1_{\{\tau > s \}} ]$, $s > t.$
\begin{proposition}
Suppose $G$ is the Green's function of the operator ${\mathcal L}.$
Then 
\begin{equation}\label{e:expginG}
{\mathbb E}_{(x,t)}[ g(X_s) 1_{\{\tau > s \}} ]= \int_{D_s} g(y) G(x,t,y,s) dy,
\end{equation}
for any bounded continuous function $g$.
\end{proposition}
\begin{proof}
If $g$ has compact support in $D_s$, we know by the definition of $G$ that
\[
u(x,t)= \int_{D_s} g(y) G(x,t,y,s) dy,
\]
is a smooth solution of ${\mathcal L}u =0$ that is continuous in $\overline{D}|_{[0,s]}$ with $u=0$ on $S$
and $u=g$ on $D_s.$ Ito's formula applied to $u(X_t,t)$ gives \eqref{e:expginG}.
Thus it only remains to treat the case when $g$ doesn't have compact support in $D_s.$ Let $g_n$ be a sequence of
continuous functions with compact support in $D_s$ converging up to $g$. Then
\[
{\mathbb E}[g(X_s) 1_{\{\tau > s \}} ]=
\lim_{n\rightarrow \infty} {\mathbb E}[g_n(X_s) 1_{\{\tau > s \}} ]
+ {\mathbb E}[g(X_s) 1_{\{\tau > s \}} {\bm 1}_{\partial D_s}(X_s)].
\]
The assumptions made on $a$ and $b$ imply that $X_s$ has a density in ${\mathbb R}^n$ and in particular
the second expectation above is $0$. Therefore:
\begin{align*}
{\mathbb E}[g(X_s) 1_{\{\tau > s \}} ]&= \lim_{n\rightarrow \infty} {\mathbb E}[g_n(X_s) 1_{\{\tau > s \}} ]\\
&=\lim_{n\rightarrow \infty} 
 \int_{D_s} g_n(y) G(x,t,y,s) dy
=
 \int_{D_s} g(y) G(x,t,y,s) dy,
\end{align*}
where the last equality follows from the bounded convergence theorem.

\end{proof}
Setting $g=1$ in \eqref{e:expginG} we get the following formula for ${\mathbb P}_{(x,t)}( \tau  > s)$:
\[
{\mathbb P}_{(x,t)} (\tau > s) = \int_{B_T}  G(x,t,y,s) dy;
\]
The density of the exit time $\tau$ is then
\begin{equation}\label{e:derofdistribution}
-\frac{\partial }{\partial s} \int_{D_s}  G(x,t,y,s) dy,
\end{equation}
whenever this derivative exists. 
When the domain $D_t$ is constant, i.e., when $D_t = D_0$ for all $t$, the above derivative is simply
\begin{equation}\label{e:derofF2}
-\frac{\partial }{\partial s} \int_{D_0}  G(x,t,y,s) dy,
=
-\int_{D_s}  G_s(x,t,y,s) dy = -\int_{D_0}  G_s(x,t,y,s) dy,
\end{equation}
whenever $G_s$ exists and is continuous (by differentiation under the integral sign, see, e.g. \cite{apostol}).
Its computation in the presence of a time dependent domain $D_t$ is known as the Leibniz formula or the ``Reynolds Transport Theorem''
\cite{flanders1973differentiation,cortez2013pdes}. All of the statements of this formula we have come across in the literature
assume that the domain $D_t$ is given as the image of a smooth flow ${\bm x}(\cdot,t): D_0 \mapsto D_t.$  Assume for now
$D_t$ can be represented as the image of $D_0$ under a smooth flow ${\bm x}$ and let $v$ denote the vector field
defined by the flow (see the paragraph following Lemma \ref{l:contGs} below for comments on the flow representation of $D_t$).
Leibniz formula given in \cite{flanders1973differentiation, cortez2013pdes} implies:
\begin{equation}\label{e:derofF}
-\frac{\partial }{\partial s} \int_{D_s}  G(x,t,y,s) dy,
=
\int_{D_s}  G_s(x,t,y,s) dy +  \int_{\partial D_s}  G(x,t,y,s) \langle v, N\rangle dS,
\end{equation}
where $N$ is the unit vector field on $\partial D_s.$
A comparison of this with \eqref{e:derofF2} shows that
the second term in \eqref{e:derofF} is the additional term arising from the fact that $D_t$ varies in time.
But by its construction  the Green's function $G$ is $0$ on $\partial D$ (\cite[Corollary 1, page 83]{friedman}), 
therefore this additional term is in fact $0$!
Then in the computation of the density of $\tau$, allowing
the domain to vary in time doesn't have a direct impact on the density formula, (i.e, the formula \eqref{e:derofF2} works
both for time dependent domains as well as those that are independent of time).

Second observation about \eqref{e:derofF}: for the derivative 
\eqref{e:derofdistribution} to exist we need the partial derivative
of $G$ with respect to $s$.
We know by \cite[Theorem 16, page 82]{friedman} that $G$ is differentiable in its $t$ and $x$ variables. But this result
does not directly address the smoothness of $G$ in the $s$ variable. One way to get smoothness of $G$ in the $s$ variable
is to work with the Green's function $G^*$ of the adjoint operator ${\mathcal L}^*$ defined as follows:
\[
{\mathcal L}^*u \doteq \langle a, Hu \rangle + \langle b^*, \nabla_x u \rangle + c^*u -\frac{\partial u}{\partial t} = 0,
\]
where
\begin{equation}\label{d:bstar}
b^*_i = -b_i + 2 \sum_{j=1}^n \frac{\partial a_{i,j}}{\partial x_j}, c^* = -\sum_{i=1}^n \frac{\partial b_i}{\partial x_i}
+ \sum_{i,j=1}^n \frac{\partial^2 a_{i,j}}{\partial x_i \partial x_j}.
\end{equation}
For $G^*$ to exist and be smooth in its $x$ and $t$ variables it suffices that $b^*$ and $c^*$ be uniformly Holder continuous
(the uniform ellipticity of $a$ is already assumed).

\begin{lemma}\label{l:contGs}
Let $b^*_i$ and $c^*$ of \eqref{d:bstar} be uniformly Holder continuous. Then $G$ is differentiable in $s$ with a continuous derivative $G_s.$
\end{lemma}
\begin{proof}
The assumptions on $b^*_i$ and $c^*$ imply that the adjoint operator ${\mathcal L}^*$ satisfies the conditions of
\cite[Theorem 16, page 82]{friedman} which says that ${\mathcal L}^*$ has associated with it a Green's function $G^*$
that is differentiable in $t$ with a continuous derivative $G^*_t.$
By \cite[Theorem 17, page 84]{friedman} $G$ and $G^*$ are dual, i.e.,
\[
G(x,t,y,s) = G^*(y,s,x,t);
\]
this and the $G_s = G^*_t$ imply the statement of the lemma.
\end{proof}

Even though in the end it has no influence on the final expression of the density,
we need the existence of a continuously differentiable flow 
${\bm x}$ that generates the domain $D$ to 1) invoke Leibniz rule and 2) to show that the resulting
density is continuous.
Many papers working on PDE with time
dependent domains use this assumption \cite{cannarsa1990damped,burdzy2004heat,cortez2013pdes}. 
Friedman's classical
book \cite{friedman} on parabolic PDE,
on which most of the arguments above are based, does not contain this
assumption directly. However, the assumptions already made on $D$ do indeed imply
that $D_t$ can be represented as the forward image of $D_0$ under a smooth flow ${\bm x}.$ To find
such a flow one can proceed as follows: first use the local graph representation of $\partial D$ given
in Assumption \ref{as:ebar} to define a flow on $\partial D$ as follows: 
\[
{\bm x}(x,t) = (h(x_2,x_3,...,x_d,t), x_2,x_3,...,x_d,t),
\]
where this definition is made in a neighborhood of $(x_0,t_0) \in \partial D$ where the graph of $h$
represents a portion of $\partial D.$ That $h$ is $C^1$ implies that ${\bm x}$ defined as above is a smooth
flow on $\partial D.$ One can now extend this flow to all of ${\mathbb R}^d$ using classical results
on the possibility of such an extension (see e.g., \cite[page 584]{burdzy2004heat} or 
\cite[page 201, Extension lemma for vector fields on submanifolds]{lee2013smooth}). 
That $D_t$ is the forward image of $D_0$
now follows from the fact that ${\bm x}$, by its definition, leaves $\partial D$ invariant and the existence uniqueness theorem for ODE.

We can now make a precise statement about the density of $\tau$:
\begin{proposition}\label{p:densityformula}
Suppose $a$ is uniformly elliptic and $a$, $b$, $b^*$ and $c^*$ are uniformly Holder continuous.
and let $D$ satisfy the assumptions \ref{as:ebar}. Then the Green's function $G$ is continuously differentiable in $s$
and the exit time $\tau$ has continuous density
\[
f^\tau(x,t,s)= -\int_{D_s}  G_s(x,t,y,s) dy, s \in (t,T].
\]
\end{proposition}
\begin{proof}
The existence and continuity of $G_s$ follows from Lemma \ref{l:contGs}; the density formula follows
from Leibniz's rule and $G=0$ on $\partial D_t$, as discussed above.
The continuity of the density follows from the continuity of $G_s$ and the fact that $D_t$ is the smooth image of
$D_0$ under the flow ${\bm x}.$
\end{proof}

\section{Conclusion}\label{s:conclusion}
The present work finds solutions to BSDE \eqref{eq:bsde} with a superlinear driver with
 singular terminal values of the form ${\bm 1}_A$, $A \in {\mathcal F}_T$. In studying this question it greatly generalizes
the class of events $A$, the assumptions on the driver $f$ as well as the filtration ${\mathcal F}_T$ as compared to the 
previous work \cite{krus:popi:seze:18}, which focused on a deterministic $f$, the filtration generated by a Brownian motion
and $A$ of the form $\{\tau_0 \le T\}$ and $\{\tau_0 > T\}$ where $\tau_0$ is the first exit time of the Brownian motion
from a fixed interval. 
With the results of Section \ref{s:secondcase} we see that under general conditions on the driver and the filtration,
the BSDE \eqref{eq:bsde} with terminal condition ${\bm 1}_A \cdot \infty$ can be solved continuously
for any $A \in {\mathcal F}_T$ that can be written as the limit
of a decreasing sequence of adapted events. The arguments in Section\ref{s:firstcase} imply that for events the form $\{\tau \le T\}$, 
where $\tau$ is a stopping time to obtain continuous solutions
to the BSDE we only need that $\tau$ has a bounded density. Our results in Section \ref{s:density} show that exit times of 
multidimensional Markovian diffusions form time dependent smooth domains satisfy this condition. Despite these generalizations 
the identification of all events $A$ in ${\mathcal F}_T$ for which the BSDE \eqref{eq:bsde} with terminal condition $\infty \cdot {\bm 1}_A$ 
has a continuous solution remains an open problem.
As already noted we rely on the density of $\tau$ in dealing with the event
$A=\{\tau \le T\}$; this reliance brings with it the assumption $p > 2$ when dealing with the terminal condition ${\bm 1}_{A} \cdot \infty.$
To remove this assumption is an open problem for future research.

Another natural direction for future research is the derivation of density formulas for exit
times for more general multidimensional processes, including those with jumps. Once such formulas are available the arguments in
Section \ref{s:firstcase} would imply the existence of continuous solution to BSDE \eqref{eq:bsde} with terminal conditions defined
by these exit times.

All results obtained in this paper can be generalized to the case where the compensator of $\pi$ is random and equivalent to the measure $\mu\otimes dt$ with a bounded density for example (see the introduction of \cite{bech:06} for example). Nevertheless since we refer to \cite{krus:popi:17,krus:popi:14} for the existence and uniqueness of the solution of BSDE, we keep this setting for $\pi$.

\appendix

\section{Proof of the upper bound (\ref{eq:a_priori_estimate_Y_L}) }  

Let us recall the arguments of the proof of \cite[Proposition 2]{krus:popi:15}. For any $k \geq 0$ we consider the BSDE 
\begin{equation*} \label{eq:truncated_bsde}
dY^k_t = -  f^k(t,Y^k_t,Z^k_t,\psi^k_t) dt + Z^k_t dW_t +  \int_\cE \psi^k_t(e) \tpi(de,dt) +dM^k_t
\end{equation*}
with bounded terminal condition $Y^k_T = \xi \wedge k$ and where
\begin{equation}\label{eq:generator_trunc_BSDE}
f^k(t,y,z,\psi) = ( f(t,y,z,\psi)-f^0_t ) + f^0_t \wedge k.
\end{equation}
The solution $Y^k$ is non negative in our setting. 
We also consider the driver
\begin{eqnarray*}
h(t,y,z,\psi) & = &b^k_t - p\frac{1}{T-t} y + [f(t,0,z,\psi)-f^0_t].
\end{eqnarray*}
with $b^k_t =  \frac{((p-1)\eta_t)^p}{(T-t)^p} + (f^0_t \wedge k)$.
Let $\eps > 0$ and denote by $(\cY^{\eps,k},Z^{\eps,k},\phi^{\eps,k},N^{\eps,k})$ the solution process of the BSDE on $[0,T-\eps]$ with driver $h$ and terminal condition $\cY^{\eps,k}_{T-\eps}=Y^{k}_{T-\eps} \geq 0$. Recall that from \ref{A3} and \ref{A4}
$$f(t,0,z,\psi)-f^0_t \leq \beta^{z,\psi}_t z+ \int_\cE \psi(e) \kappa^{0,0,\psi,0}_t(e)  \mu(de),$$
where
$$\beta^{z,\psi}_t = \frac{f(t,0,z,\psi)-f(t,0,0,\psi)}{z \mathbf 1_{z \neq 0}}.$$
From \ref{A4}, $\beta^{z,\psi}$ is a bounded process by $L$.
Hence by a comparison argument with the solution for linear BSDE (see \cite{quen:sule:13}, Lemma 4.1) we have
$$\cY^{\eps,k}_t \leq \bE \left[ \Gamma_{t,T-\eps} Y^{k}_{T-\eps} + \int_t^{T-\eps} \Gamma_{t,s} b^k_s ds \bigg| \cF_t\right]$$
where for $t \leq s \leq T-\eps$
\begin{eqnarray*}
\Gamma_{t,s} & = & \exp\left( -\int_t^s \frac{p}{T-u} du + \int_t^s \beta^{Z^{\eps,k},\phi^{\eps,k}}_u dW_u - \frac{1}{2} \int_t^s (\beta^{Z^{\eps,k},\phi^{\eps,k}}_u)^2 du\right)V^{\eps,k}_{t,s} \\
&=& \left( \frac{T-s}{T-t} \right)^{p}  \exp\left(  \int_t^s \beta^{Z^{\eps,k},\phi^{\eps,k}}_u dW_u - \frac{1}{2} \int_t^s (\beta^{Z^{\eps,k},\phi^{\eps,k}}_u)^2 du\right)V^{\eps,k}_{t,s}
\end{eqnarray*}
and
\begin{equation}\label{eq:auxil_proc_V}
V^{\eps,k}_{t,s} = 1+ \int_t^s \int_{\cZ}  V^{\eps,k}_{t,u^-} \kappa^{0,\phi^{\eps,k},0}_u(z) \tpi(dz,du).
\end{equation}
Hence
$$\cY^{\eps,k}_t \leq \frac{1}{(T-t)^p} \bE \left[\eps^\rho V^{\eps,k}_{t,T-\eps} Y^{k}_{T-\eps} + \int_t^{T-\eps} V^{\eps,k}_{t,s} (T-s)^p b^k_s  ds \bigg| \cF_t\right].$$
Since $b^k \geq 0$ it holds that $\cY^{\eps,k}_t  \geq 0$ a.s. for every $t \in [0,T]$. Hence from Condition \ref{C1}
$$f^k(t,\cY^{\eps,k}_t,Z^{\eps,k}_t,\phi^{\eps,k}_t) \leq - \frac{1}{\eta_t } (\cY^{\eps,k}_t)^{q}+ f^k(t,0,Z^{\eps,k}_t,\phi^{\eps,k}_t).$$
It follows that
\begin{eqnarray*}
f^k(t,\cY^{\eps,k}_t,Z^{\eps,k}_t,\phi^{\eps,k}_t) & \leq & h(t,\cY^{\eps,k}_t,Z^{\eps,k}_t,\phi^{\eps,k}_t)-\frac{1}{\eta_t } (\cY^{\eps,k}_t)^{q} -\frac{((p-1)\eta_t)^p}{(T-t)^p} +\frac{p}{T-t} \cY^{\eps,k}_t \\
&\leq& h(t,\cY^{\eps,k}_t,Z^{\eps,k}_t,\phi^{\eps,k}_t),
\end{eqnarray*}
where we used the Young inequality: $c^p+(p-1)y^q-p cy\ge 0$ which holds for all $c,y\ge 0$. The comparison theorem implies $Y^{k}_t \leq \cY^{\eps,k}_t$ for all $t\in[0,T-\eps]$ and $\eps > 0$.

Recall once again from Condition \ref{C3}, then $V^{\eps,L}_{t,.}$ belongs to $\bH^\varpi(0,T-\eps)$ for some $\varpi \geq 2$.
From the upper bound $Y^{k}_t  \leq k(T+1)$ and from the integrability property of $V^{\eps,k}_{t,.}$, with dominated convergence, by letting $\eps \downarrow 0$ we obtain a.s.
$$\bE \left[\eps^p V^{\eps,k}_{t,T-\eps} Y^{k}_{T-\eps}  \bigg| \cF_t\right] \longrightarrow 0.$$
From Assumption \ref{C3}, by the proof of Proposition A.1 in \cite{quen:sule:13}, there exists a constant $K_{\vartheta,L,\ell'}$ such that a.s.
$$\bE \left[\int_t^{T-\eps} (V^{\eps,k}_{t,s})^{\frac{\ell'}{\ell'-1}}  ds\bigg| \cF_t\right] \leq (K_{\vartheta,L,\ell'})^{(\ell'-1)/\ell'}.$$
From Conditions \ref{C2} and \ref{C4}, it follows that the process $((T-t)^p b^k_t, \ 0\leq t \leq T)$ belongs to $\bH^{\ell'}(0,T)$ for any $1 < \ell' \leq \ell$. Therefore by H\"older's inequality we obtain
\begin{eqnarray*}
&& \bE \left[\int_t^{T-\eps} V^{\eps,k}_{t,s}(T-s)^p b^k_s  ds \bigg| \cF_t\right] \leq K_{\vartheta,L,\ell'} \bE \left[\int_t^{T} ((T-s)^p b^k_s)^{\ell'}  ds \bigg| \cF_t\right]^{1/\ell'}.
\end{eqnarray*}
Hence we can pass to the limit as $\eps \downarrow 0$
\begin{eqnarray*}
Y^{k}_t & \leq & \frac{K_{\vartheta,L,\ell'}}{(T-t)^p} \bE \left[ \ \int_t^{T}  ((T-s)^p b^k_s)^{\ell'} ds \bigg| \cF_t\right]^{1/\ell'}.
\end{eqnarray*}
Assumptions \ref{C2} and \ref{C4} imply by monotone convergence for $k \to \infty$
\begin{eqnarray*}
Y^{k}_t & \leq &  \frac{K_{\vartheta,L,\ell'}}{(T-t)^p} \bE \left[ \ \int_t^{T} \left( ((p-1)\eta_s)^p + (T-s)^p (f^0_s) \right)^{\ell'} ds \bigg| \cF_t\right]^{1/\ell'} <+\infty
\end{eqnarray*}
Using again H\"older's inequality for the conditional expectation, we obtain the upper bound in \eqref{eq:a_priori_estimate_Y_L}.

\bibliography{biblio}

\end{document}